\documentclass[12pt, leqno]{amsart}
\usepackage[english]{babel}
\usepackage[utf8]{inputenc}
\usepackage[tmargin=2.0cm, bmargin=2.0cm, lmargin=3.0cm, rmargin=3.0cm]{geometry}
\usepackage{amsmath}
\usepackage{amsthm}
\usepackage{hyperref}
\usepackage{float}
\usepackage{subfigure}
\usepackage{verbatim}
\usepackage{amsfonts}
\usepackage{graphicx}
\usepackage{enumerate}
\usepackage[colorinlistoftodos]{todonotes}
\newtheorem{theorem}{Theorem}
\newtheorem{prop}[theorem]{Proposition}
\newtheorem{lemma}[theorem]{Lemma}

\theoremstyle{definition}

\theoremstyle{plain}
\newtheorem{remark}[theorem]{Remark}
\theoremstyle{plain}
\newtheorem{fact}[theorem]{Fact}

\newcommand{\p}{\mathbb{P}}

\newcommand{\WW}{\mathbb{W}}
\newcommand{\WWq}{\mathbb{W}^{(q)}}
\newcommand{\WWqprime}{\WW^{(q)'}}

\newcommand{\qscale}{W^{(q)}}

\newcommand{\qscaleprime}{W^{(q) \prime}}

\newcommand{\qscaleY}{\mathbb W^{(q)}}
\newcommand{\qscaleprimeY}{\mathbb{W}^{(q)'}}

\newcommand{\wq}{w^{(q)}}
\newcommand{\wqprime}{w^{(q)'}}

\newcommand{\vq}{V^{(q)}}
\newcommand{\vqprime}{V^{(q)\prime}}

\newcommand{\vqint}{\int_{0}^{\infty} \wq (x;-z) \frac{z}{r} \mathbb{P}(X_{r} \in dz)}

\newcommand{\valueP}{\upsilon^{\kappa^{r}}_{c_{1}^{*},c_{2}^{*}}}
\newcommand{\qplus}{q^{+}(q)}
\newcommand{\qminus}{q^{-}(q)}
\newcommand{\aplus}{A^{+}}
\newcommand{\aminus}{A^{-}}
\newcommand{\qplusY}{q^{+}_{Y}(q)}
\newcommand{\qminusY}{q^{-}_{Y}(q)}
\newcommand{\aplusY}{A^{+}_{Y}}
\newcommand{\aminusY}{A^{-}_{Y}}
\def\beq{\begin{eqnarray}} \def\eeq{\end{eqnarray}}

\def\al*#1{\begin{align*}#1\end{align*}}

\def\ga*#1{\begin{gather*}#1\end{gather*}}

\def\alat*#1#2{\begin{alignat*}{#1}#2\end{alignat*}}
\def\bea{\begin{eqnarray*}}
\def\eea{\end{eqnarray*}}
\def\ml*#1{\begin{multline*}#1\end{multline*}}

\title[Impulse control problem]{Optimality of impulse control problem in refracted L\'evy model with Parisian ruin and transaction costs}
\thanks{I. Czarna is partially supported by the National Science Centre Grant No. 2015/19/D/ST1/01182.\\
A. Kaszubowski is partially supported by the National Science Centre Grant No. 2015/17/B/ST1/01102.}
\date{{\small \today}}

\author[Irmina Czarna]{Irmina Czarna$^{*}$}
\address{$*$ Faculty of Pure and Applied Mathematics, Wroc\l aw University of Science and Technology, ul. Wybrze\.ze Wyspia\'nskiego 27, 50-370 Wroc\l aw, Poland}
\email{irmina.czarna@pwr.edu.pl}

 \author[Adam Kaszubowski]{Adam Kaszubowski$^{\dag}$}
\address{$\dag$ Mathematical Insititute, University of Wroc\l aw, pl. Grunwaldzki 2/4, 50-384 Wroc\l aw, Poland}
\email{adam.kaszubowski@math.uni.wroc.pl}

\begin{document}
\begin{abstract}
In this paper we investigate an optimal dividend problem with transaction costs, where the surplus process is modelled by a refracted L\'evy process and the ruin time is considered with Parisian delay. Presence of the transaction costs implies that one need to consider the impulse control problem as a control strategy in such model. 
An impulse policy $(c_1,c_2)$, which is to reduce the reserves to some fixed level $c_1$ whenever they are above another level $c_2$ is 
an important strategy for the impulse control problem. Therefore, we give sufficient conditions under which the above described impulse policy is optimal.
Further, we give the new analytical formulas for the Parisian refracted $q$-scale functions in the case of the linear Brownian motion and the Cr\'amer-Lundberg process with exponential claims. Using these formulas we show that for these models there exists a unique $(c_1, c_2)$ policy which is optimal for the impulse control problem. Numerical examples are also provided. 
\vspace{3mm}


\end{abstract}
\maketitle 
 \vspace{3mm}
\section{Introduction}
For many years applied mathematicians has been trying to create the models that allow to describe reality in the terms of mathematics. A special role is played by models used to describe phenomena that develop over time and in which there is a some factor of randomness. In such case, it is important to approximate some certain characteristics or to find some event probabilities. For example, insurance companies need to estimate the amount of reserves that will allow them to be solvent with a very high probability. In this case, the question arises about the size of these reserves. Another example can be hedging companies, which, when valuing financial instruments, often use stochastic models. 

\hspace{0.6cm} In this paper, we focus on another classic problem affecting the companies, namely on the issue of the optimal dividend payments. Dividend is the transfer of a certain portion of the company's finances to the investors, so it is one of the tools for shareholders to receive the profits from the company's support. Dividends also may attract new investors to the company, and thus provide further financing. On the other hand, significant dividend payments result in the reduction of the company funds, and thus may lead to a significant increase of the probability of losing liquidity. Therefore dividend payments must  be made in an optimal way, and they may be made up to the company's bankruptcy. The problem of the bankruptcy is related to the ruin theory, traditionally considered in the context of the insurance companies, where the ruin moment is related to the process of financial surplus, and in particular with its size at a given moment. The classically defined moment of ruin is the first moment when the surplus process goes below the level zero. Nowadays, such a moment of ruin has a small chance of occurrence. The reason for that is the fact that such companies control (and are controlled) so that the probability of such event stays at a very low level, whether through the impact of the additional cash or prior fixing of the financial reserves. Probability of the ruin as well as the theory of ruin plays therefore a different role. It is a determinant of the financial situation of the company, which allows to make strategic decisions in the company management. Therefore, it is important to investigate different definitions of the ruin and choose the proper one for our case. 

\hspace{0.6cm} The classical ruin seems to be an intuitively obvious definition and if such a moment comes, we expect that the company will immediately declare bankruptcy. In practice, very often the investors or the government try to save the company from bankruptcy. Additionally, the too restrictive definition of the ruin as an economic determinant causes freezing of too much cash securing, and so the company grows weaker. Analysing this problem in the terms of the above doubts, it can be concluded that there is a natural need for a different definition of the ruin and in particular the separation of the technical ruin, i.e. exceeding the zero level, and the actual moment of the bankruptcy announcement. For this reason, many alternatives appeared in the literature, for example the so-called Parisian ruin model, which is considered in this paper.
In this approach we say that the company announces bankruptcy if the risk process to goes below zero (or to the so-called \textit{red zone}) and stays there longer than a certain fixed time $r> 0$. Such \textit{Parisian} stopping times have been studied by Chesney et al. \cite{CJPY97} in the context of barrier options in mathematical finance. In another paper Czarna and Palmowski \cite{CP11} gave the first description of the Parisian ruin probability for a general spectrally negative L\'evy processes.

\hspace{0.6cm} Now, let us define a class of the processes that are usually used to model the financial surplus. 
One of the most known stochastic processes used in the theory of ruin is the Cr\'amer-Lundberg process, which can be presented in the following form
\[
X_t = x + pt - \sum_{i=1}^{N_t} U_i,
\]
where $x \geq 0$ represents the initial capital, $p> 0$ is the constant intensity of the premium income, $\lbrace N_t \rbrace_{t\geq 0}$ is a homogeneous Poisson process with the intensity $\lambda > 0$ and $\lbrace U_i \rbrace_{i=1}^{\infty}$ are positive i.i.d random variables. Therefore, the compound Poisson the claims of the customers. 
The form of the Cr\'amer-Lundberg process has some benefits in the aspect of ease of calculations, however, sometimes it may turn out to be too far-reaching simplification. For example, one can see that between successive claims this process is deterministic, so it does not take into account certain market fluctuations. In addition, in the form of this process, we are not able to distinguish large and small claims, and what is done in practice for insurance companies. Therefore, one can consider a wider class of spectrally negative L\'evy processes, which contain the Cr\'amer-Lundberg process. This class of the processes include also linear Brownian motion, Cauchy process and $\alpha$-stable processes. 

\hspace{0.6cm} As was mentioned before, one would like to distinguish the moment of exceeding the zero level with the actual bankruptcy by considering the Parisian ruin time. However, to further approximate the model to the reality, we will add an additional assumption. Namely, when the surplus process is in the \textit{red zone} (i.e. below zero) we assume that it receive a steady flow of cash with the intensity of $\delta> 0$, until it will reach the positive values. It reflects saving the company from bankruptcy by investors or government. In order for such assumption to be added, one need to use the so-called spectrally negative \textit{refracted} L\'evy process, which was introduced by Kyprianou and Loeffen \cite{KL10}. 

\hspace{0.6cm} Therefore, using the above mentioned assumptions, our goal is to analyse the problem of the optimal dividend payments, where each payment is be accompanied by a certain fixed transaction fee $\beta > 0$. 

\hspace{0.6cm} Historically, many papers have been written on this topic. The first problem was examined by de Finetti in \cite{F57}. He postulated that if the risk process behaves like a random walk with the increments of $\pm$ 1, then the optimal dividend strategy is of barrier type. The barrier strategy is that the company pays everything above a certain fixed level $b$. The next step was to consider the continuous-type processes. In the framework of the linear Brownian motion process as well as the Cramer-Lundberg process, a similar result was obtained, i.e. the optimal strategy is the barrier strategy, see \cite{GS04}, \cite{GS06} and \cite{JS95}. Finally, in \cite{APP07}, the optimal barrier strategy for the entire class of spectrally negative L\'evy processes was examined. The authors received certain conditions that would ensure that this strategy is the optimal one. Moreover, they expressed these conditions in the language of the so-called scale functions, which will be introduced in Subsection \ref{Exit problems and scale functions}. Note that in the all above-mentioned papers there were no transaction cost. The only exception is \cite{L08}, where this assumption was made, also for the class of the spectrally negative L\'evy processes. Due to this assumption, further consideration of the barrier strategy was not possible, thus it was replaced by an impulse control strategy, which will be described in details in the following chapters. Moreover, in Section \ref{optimality}, the sufficient conditions were obtained providing the optimality of this dividend strategy.

\hspace{0.6cm} The rest of the paper is organized as follows. First, we introduce some basic notation and definitions related to the spectrally negative L\'evy processes and the refracted counterpart.
In particular, we will introduce the scale functions and explain why in this theory they play the key role. In Subsection \ref{Dividend} we will describe the dividend problem and will explain what means that the dividend strategy is the optimal strategy. Section \ref{Impulse strategy with the Parisian ruin} is the main part of this paper. We will introduce there the impulse $(c_1,c_2)$ policy and we will provide sufficient conditions that the derivative of Parisian refracted scale must fulfils to ensure that the strategy is optimal. The last part of this paper is an examples section, where we will give the new analytical formulas for the Parisian refracted scale functions in the case of the linear Brownian motion and the Cr\'amer-Lundberg process with exponential claims. Using these formulas we will show that for these models there exists a unique impulse policy which is optimal for the impulse control problem. Numerical examples will be also provided. 

\section{Mathematical model}
\subsection{Surplus process}
Let  $(\Omega,\mathcal{F},\mathbb{F} = \lbrace \mathcal{F}_{t}: t \geq 0 \rbrace$, $\mathbb{P})$ be the probability space which satisfy usual conditions. On this probability space we consider process $X = \lbrace X_{t}\rbrace_{t\geq 0}$ being a spectrally negative L\'evy process, namely the stochastic process issued from the origin which has stationary and independent increments and c\`adl\`ag paths that have no positive jump discontinuities. To avoid degenerate cases, we exclude the case where $X$ has monotone paths. As a strong Markov process we shall endow $X$ with probabilities $\{\mathbb P_x:x\in\mathbb R\}$ such that under $\mathbb P_x$ we have $X_0=x$ with probability one. Further $\mathbb E_x$ denotes expectation with respect to $\mathbb P_x$. Recall that $\mathbb P=\mathbb P_0$ and $\mathbb E=\mathbb E_0$. 
 Every spectrally negative L\'evy process can be represented by the triple $(\gamma,\sigma,\Pi)$ where  $\gamma \in \mathbb{R}$,  $\sigma \geq 0 $ and $\Pi$ is a measure on $(-\infty,0)$ which satisfies
\[
\int_{(-\infty,0)} (1\wedge x^2) \Pi(dx) < \infty.
\]
The Laplace exponent of $X$ is defined through
\[
\psi(\theta) := \log(\mathbb{E}[e^{\theta X_1}]) = \gamma \theta + \frac{1}{2} \sigma^2 \theta^2 + \int_{(-\infty,0)} \Bigl( e^{\theta x} -1 - \theta x \textbf{1}_{\lbrace -1 < x < 0\rbrace}\Bigr)\Pi(dx) 
\]
 for any $\theta \geq 0 $. For background on spectrally negative L\'evy processes we refer the reader to \cite{B96,K06}. 

We assume that in our model the surplus process $R$ is modelled by spectrally negative refracted L\'evy process which means that we allow injecting (in continuous way) certain amount of money with intensity $\delta>0$ when reserves are below zero. 
Namely, one can define such process as unique strong solution $R = \lbrace R_{t}\rbrace_{t \geq 0} $ to the following stochastic differential equation: 
\begin{equation}
\nonumber
dR_{t} = dX_{t} - \delta\mathbf{1}_{\lbrace R_{t} >b \rbrace} dt, \quad \text{for }\delta > 0 \quad \text{and } \quad b=0 .
\end{equation}
Note that, refracted process $R$ with $b \geq 0$ was examined before by Kyprianou and Loeffen in \cite{KL10}. As in \cite{LCR17} we focus here on the case when the refraction level $b$ equals zero. Moreover, to be compatible with \cite{KL10} and \cite{LCR17}, we subtract $\delta$ on the positive half-line instead of adding it on the negative half-line, however, the practical effect is the same. 

From the above equation it is easy to observe that above the level $b$ process $R$ evolves as process $Y_t = X_t - \delta t$. 
Since the process $Y$ is a spectrally negative L\'evy process with the L\'evy triplet $(\gamma - \delta, \sigma, \Pi)$ its Laplace exponent is given by 
\[
\psi_{Y}(\theta) = \psi(\theta) - \delta \theta,
\]
In particular, process $Y$ retains the probabilistic properties of the process $X$, e.g. the bounded/unbounded variation of the paths. 
Moreover, we want to emphasize here that process $R$ is no longer spatial homogeneous which means that it is not a L\'evy process. In Section \ref{optimality} we will prove that process $R$ is a Feller process and we will present the form of its infinitesimal generator.

\subsection{Dividend problem}\label{Dividend}
Let us now formally introduce the problem studied in this paper, in particular we define the optimization criterion, and then define the candidate for the optimal strategy.
Denote $\pi$ as a dividend or control strategy, where $\pi = \lbrace L_{t}^{\pi}\rbrace_{t \geq 0 }$ is a non-decreasing, left-continuous $\mathbb{F}$-adapted process which starts at zero. We will assume that process $L^{\pi}$ is a pure jump process, i.e.
\begin{equation}\label{pureJump}
L_{t}^{\pi} = \sum_{0 \leq s < t} \Delta L_{s}^{\pi}, \quad \text{for all } t \geq 0.
\end{equation}
Here we mean by $\Delta L_{s}^{\pi} = L_{s+}^{\pi} - L_{s}^{\pi}$ the jump of the process $L^{\pi}$ at time s.
Therefore, random variable $L_{t}^{\pi}$ can be interpreted as a cumulated dividends to the time $t$. 
Note that, pure jump assumption is taken directly from the presence of non-zero transaction costs and such control strategies as \eqref{pureJump} are known as \textit{impulse controls}. 
Let us define the controlled risk process $U^{\pi} = \lbrace{ U_{t}^{\pi}\rbrace}_{ t \geq 0}$ by the dividend strategy~$\pi$: $$U_{t}^{\pi} := R_{t} - L_{t}^{\pi}.$$
The company pays dividends up to its bankruptcy moment which in our model is the Parisian ruin time. Let us formally define it as 

$$\kappa^{r} = \inf \lbrace t > 0: t - \sup\lbrace s < t : U^{\pi}_{s} \geq 0  \rbrace \geq r, U^{\pi}_{t} < 0 \rbrace,$$

where $r > 0$ is the so-called Parisian delay.  

Let us define the value function of a dividend strategy $\pi$:
\begin{equation}
\nonumber
\upsilon_{\pi}^{\kappa^{r}}(x) = \mathbb{E}_{x}\left[ \int_{0}^{\kappa^{r}}e^{-qt} d \Bigl(L_{t}^{\pi} - \sum_{0\leq s < t} \beta \mathbf{1}_{\lbrace \Delta L_{s}^{\pi} > 0 \rbrace} \Bigr)\right], \quad \text{ for } x \geq 0,
\end{equation}
where $q > 0$ is the discount rate and $\beta > 0$ denotes the transaction cost which occurs whenever the company pays dividends. Since (\ref{pureJump}) is assumed, the above integral can be interpreted as the following sum 
\begin{equation}
\nonumber
\upsilon^{\kappa^{r}}_{\pi}(x) = \mathbb{E}_{x}\left[\sum_{0 \leq t < \kappa^{r}} e^{-qt}\Bigl(\Delta L_{t}^{\pi} - \beta \mathbf{1}_{\lbrace \Delta L_{t}^{\pi} > 0 \rbrace}\Bigr)\right], \quad x \geq 0. 
\end{equation}
We call a strategy $\pi$ admissible if we do not get to the {\it red zone} due to dividend payments,~i.e. 
\begin{equation}\label{red zone}
U_t^{\pi} - \Delta L_t^{\pi} \geq 0, \quad \textrm{ for } \quad t < \kappa^r.
\end{equation} 
 Let $\mathcal{A}$ be the set of all admissible dividend strategies. Our main goal is to find the optimal value function $\upsilon_{*}$ given by
\begin{equation}
\nonumber
\upsilon_{*}(x) = \sup_{\pi \in \mathcal{A}} \upsilon^{\kappa^{r}}_{\pi}(x)
\end{equation}
and the optimal strategy $\pi_{*} \in \mathcal{A}$ such that
\begin{equation}
\nonumber
\upsilon_{\pi_{*}}^{\kappa^{r}}(x) = \upsilon_{*}(x), \quad \text{for all } x \geq 0. 
\end{equation}

\subsection{Exit problems and scale functions}\label{Exit problems and scale functions}
In this section we introduce key tools that will allow the optimality of dividend strategy to be investigated.
From the application point of view, one of the most important issues studied in the theory of L\'evy processes are so-called exit problems. The classical de Finetti dividend problem can also be expressed using exit identities, therefore we will recall here basic results from this topic. 

First, for a $\in \mathbb{R}$, we define the following first-passage stopping times
\begin{equation}
\nonumber
\begin{split}
&\tau_a^{-} = \inf\lbrace t > 0: X_t < a\rbrace \quad \text{and} \quad \tau_{a}^{+} = \inf\lbrace t > 0: X_t \geq a \rbrace, \\
& \nu_{a}^{-} = \inf\lbrace t > 0 : Y_t < a \rbrace \quad \text{and} \quad \nu_{a}^{+} = \inf\lbrace t > 0: Y_t \geq a\rbrace, \\
& \kappa_{a}^{-} = \inf \lbrace t > 0: R_t < a \rbrace \quad \text{and} \quad \kappa_a^{+} = \inf\lbrace t > 0: R_t \geq a \rbrace .
\end{split}
\end{equation}
One can be interested in obtaining an analytical representation of the following expression (the so-called two-sided exit problem) 
\begin{equation}
\nonumber
\mathbb{E}_{x}\Bigl[e^{-q \tau_{c}^{+}}\mathbf{1}_{\lbrace \tau_{c}^{+} < \tau_{0}^{-}\rbrace} \Bigr].
\end{equation}
Namely, we would like to examine a unit payment made when the level $c$ is reached before the first moment when the level zero is exceeded. This payment is additionally discounted by a discount factor  $q>0$. 
To obtain the analytical expression for the above expectation let us define the following function.

For each $q \geq 0$ there exists a function $W^{(q)}:\mathbb{R} \rightarrow [0,\infty)$, called the $q$-scale function, which satisfies $W^{(q)}(x) = 0$ for $x<0$ and is characterised on $[0,\infty)$ as a strictly increasing and continuous function whose Laplace transform is given by 
\begin{equation}
\nonumber
\int_{0}^{\infty} e^{-\theta x}W^{(q)}(x) dx = \frac{1}{\psi(\theta)-q}, \quad \text{for } \theta > \Phi(q),
\end{equation}
where $\Phi(q) = \sup\lbrace \theta \geq 0: \psi(\theta) = q\rbrace$ is the right-inverse of $\psi$. We define the second scale function by
\begin{equation}
\nonumber
Z^{(q)}(x) := 1 + q\int_{0}^{x} W^{(q)}(y)dy, 
\quad x \in \mathbb{R}.
\end{equation}
It turns out that for $-\infty < a \leq x \leq c <\infty$ and $q \geq 0$ (see e.g., \cite{K06})
\begin{equation}
\nonumber
\mathbb{E}_{x}\Bigl[e^{-q \tau_{c}^{+}}\mathbf{1}_{\lbrace \tau_{c}^{+} < \tau_{a}^{-}\rbrace} \Bigr] = \frac{W^{(q)}(x-a)}{W^{(q)}(c-a)} 
\end{equation}
and also for $q>0$
\[
\mathbb{E}_x\left[e^{-q\tau_a^{-}}\textbf{1}_{\lbrace \tau_a^{-} < \tau_c^{+}\rbrace}\right] = Z^{(q)}(x-a) - \frac{Z^{(q)}(c-a)}{\qscale (c-a)}\qscale (x-a).
\]
Analogously we can define the scale functions for L\'evy process $Y$, and we will use notation $\qscaleY$ and $\mathbb{Z}^{(q)}$ for the first and second scale functions for $Y$, respectively. 
Define the scale function for refracted process $R$ as follows: For $q \geq 0$ and $x,a \in \mathbb{R}$
\begin{equation}\label{small w}
w^{(q)}(x;a) := \begin{cases} W^{(q)}(x-a), \quad \text{ for }x < 0 \\ W^{(q)}(x-a)  + \delta\int_{0}^{x} \mathbb{W}^{(q)}(x-y) W^{(q)'}(y-a)dy,  \quad \text{ for }x \geq 0.
\end{cases}
\end{equation}
In particular, we write $w^{(q)}(\cdot;0)=w^{(q)}(\cdot)$ when $a=0$. 
One can see that the above definition differs from the definition of scale functions for $X$ and $Y$. However, in \cite{KL10} it was proved that for $-\infty < a \leq x \leq c < \infty$ and $q \geq 0$
\[
\mathbb{E}_x\left[e^{-q \kappa_c^{+}}\textbf{1}_{\lbrace \kappa_c^{+} < \kappa_{a}^{-}\rbrace}\right] = \frac{\wq(x;a)}{\wq(c;a)}.
\]
Therefore, one can see that for process $R$, function $\wq$ gives the same representation for the two-sided exit problem as scale functions $W^{(q)}$ and $\qscaleY$. 

In this paper we additionally consider Parisian ruin time, and from \cite{LCR17} it is known that 
\begin{equation}\label{Two sided exit problem Parisian}
\mathbb{E}_{x}\Bigl[e^{-q \kappa_{a}^{+}}\mathbf{1}_{\lbrace \kappa_{a}^{+} < \kappa^{r}\rbrace}\Bigr] = \frac{\vq(x)}{\vq(a)},
\end{equation}
where \begin{equation}
\nonumber
V^{(q)} (x) = \int_{0}^{\infty}w^{(q)}(x;-z)\frac{z}{r}\mathbb{P}(X_{r} \in dz).
\end{equation}
Since scale functions occur in many fluctuation identities, the natural question is if it is possible to calculate them explicitly. 
The answer is that for some particular examples like Brownian motion with drift or Cram\'er-Lundberg process with exponential jumps, the form of functions $W^{(q)}$, $w^{(0)}$, $V^{(0)}$ can be obtained explicitly (see \cite{B96,HK11,KKR12,K06,LCR17}).  
\subsection{Properties of scale functions}

In this part, we will investigate properties of the scale functions, which will be crucial for further proofs in this paper.

At the beginning, let us cover the behaviour of the scale functions at zero. Recall that $(\gamma,\sigma,\Pi)$ is a L\'evy triple of the process $X$ and set $p := \gamma + \int_{0}^{1}x\Pi(dx)$ when process $X$ is of bounded variation $($this quantity then represents drift of the process$)$. Then 
\begin{equation}\label{Scale function at 0}
W^{(q)}(0+) = \begin{cases} \frac{1}{p} & \quad \text{ when X has bounded variation paths} \\ 0 & \quad \text{ otherwise.}
\end{cases}
\end{equation}
From \eqref{small w} one can see that the initial value of $\wq$ equals $W^{(q)}(-a)$. 
Whereas, for $\vq$, one can find in \cite{LCR17, LCP13} that
\begin{equation}\label{int wq and Xr in dz}
V^{(q)} (0) = \int_{0}^{\infty} \qscale (z) \frac{z}{r} \mathbb{P}(X_{r} \in dz) = e^{qr}. \end{equation}
The initial value of $W^{(q)'}$ equals (see e.g., \cite{K06})
\[
\qscaleprime(0+)= \lim_{x \rightarrow 0^{+}} \qscaleprime(x) = \begin{cases} \frac{2}{\sigma^2} &\text{when } \sigma > 0 \\ \frac{\Pi(0,\infty) + q}{p^2} & \text{when } \sigma = 0 \text{ and } \Pi(0,\infty) < \infty \\ \infty & \text{otherwise.}\end{cases}
\]

Moreover, for $w^{(q)}$ the following proposition was proved in \cite{CPRY19}. 

\begin{prop}\label{Derivative of w exist}
In general,  $w^{(q)}(\cdot;a)$ is a.e.~continuously differentiable and its derivative is of the form
\begin{equation*}
w^{(q)'}(x;a) =\begin{cases} W^{(q)'}(x-a), \quad \text{ for }x < 0 \\(1+ \delta \mathbb{W}^{(q)}(0))W^{(q)'}(x-a)+ \delta\int_{0}^{x}\mathbb{W}^{(q)'}(x-y)W^{(q)'}(y-a)dy ,  \text{ for } x \geq 0.
\end{cases}
\end{equation*}
In particular, if $X$ is of unbounded variation, then $w^{(q)}(\cdot;a)$ is $C^1(a,\infty)$. 
 On the other hand, if we assume that $W^{(q)}(\cdot-a) \in C^1(a,\infty)$ for $X$ is of bounded variation, 
then $w^{(q)}(\cdot;a)$ is also $C^1((a,\infty) \backslash \{ 0\})$. 
\end{prop}

\section{Impulse strategy with the Parisian ruin}\label{Impulse strategy with the Parisian ruin}
Let us present the candidate to be an optimal strategy for the dividend problem described in Section \ref{Dividend}. 
Formally, define the so-called impulse strategy $\pi_{c_{1},c_{2}}$. 
Set two constants $c_1$ and $c_2$ such that $c_{2} > c_{1} + \beta$ and $c_{1} \geq 0$. 
Next, fix $\lbrace \tau_{k}^{c_{1},c_{2}} , k = 1,2,..\rbrace$ as a set of the stopping times, such that:
\begin{equation}
\nonumber
\tau_{k}^{c_{1},c_{2}} = \inf\lbrace t > 0 : R_{t} > [(R_{0} \vee c_{2}) + (c_{2} - c_{1})(k-1)]\rbrace , \quad k = 1,2,...
\end{equation}
The strategy $\pi_{c_{1},c_{2}} = \lbrace L_{t}^{c_{1},c_{2}} : t\geq 0\rbrace$ is defined as
\begin{equation}
\nonumber
L_{t}^{c_{1},c_{2}} = \mathbf{1}_{\lbrace \tau_{1}^{c_{1},c_{2}} < t \rbrace}([R_{0} \vee c_{2}]-c_{1}) + \sum_{k=2}^{\infty} \mathbf{1}_{\lbrace \tau_{k}^{c_{1},c_{2}}<t \rbrace}(c_2-c_1), \quad t\geq 0. 
\end{equation}
Then, the controlled risk process is of the form $U_{t}^{c_{1},c_{2}} = R_{t} - L_{t}^{c_{1},c_{2}}$. Note that in the terms of $U_{t}^{c_{1},c_{2}}$ one can write
$\tau_{1}^{c_{1},c_{2}} = \inf\lbrace t >0 : U_{t}^{c_{1},c_{2}} > c_{2}\rbrace$ and  $\tau_{k}^{c_{1},c_{2}} = \inf \lbrace t > \tau_{k-1}^{c_{1},c_{2}} : U_{t}^{c_{1},c_{2}} > c_{2}\rbrace$ for $k \geq 1$. Therefore, the impulse strategy is to reduce the risk process to $c_1$ whenever the process exceeds level $c_2$. It is assumed that the distance between $c_1$ and $c_2$ must be greater than $\beta$, because after paying the transaction costs there must be something left for shareholders. Additionally, the condition that $c_{1} \geq 0$ is a consequence of \eqref{red zone}. \\

Before we give necessary conditions for $(c_1,c_2)$ strategy to be optimal, we need to consider the form of the value function as a helping tool for the further investigations. 
\subsection{Representation of the value function}
\begin{prop}\label{first_representation}
The function $v_{c_{1},c_{2}}^{\kappa^{r}}$ for the strategy $\pi_{c_{1},c_{2}}$ with the ruin time $\kappa^{r}$ is of the form:
\begin{equation}\label{v_first}
v_{c_{1},c_{2}}^{\kappa^{r}}(x) = \begin{cases} (c_{2}-c_{1}-\beta)\frac{V^{(q)} (x)}{V^{(q)} (c_{2}) - V^{(q)} (c_{1})}, & \text{for } x \leq c_{2} 
\\ x-c_{1}-\beta + (c_{2}-c_{1}-\beta)\frac{V^{(q)} (c_{1})}{V^{(q)} (c_{2}) - V^{(q)} (c_{1})}, & \text{for } x > c_{2}. \end{cases}
\end{equation}
\end{prop} 
\begin{proof}
At the beginning of the proof note that it is sufficient to prove this Proposition only for $x \leq c_{2}$, because $U^{c_{1},c_{2}}$ is a Markov process and if we are above level $c_{2}$ we put the process into level $c_{1}$ immediately. Also recall the discussions that precede equality (\ref{Two sided exit problem Parisian}). \\
Assume that $x \leq c_{2}$. The first time when we paid dividends is $\tau_{1}^{c_{1},c_{2}}$ that means that we must wait until the first time when process $U^{c_{1},c_{2}}$ is greater that $c_{2}$. Using strong Markov property we have that:
\begin{equation}\label{star}
\upsilon_{c_{1},c_{2}}^{\kappa^{r}}(x) = \mathbb{E}_{x}\Bigl[e^{-q \kappa_{c_{2}}^{+}}\mathbf{1}_{\lbrace \kappa_{c_{2}^{+}}<\kappa^{r} \rbrace} \Bigr]\upsilon_{c_{1},c_{2}}^{\kappa^{r}}(c_{2}) = 
\frac{V^{(q)}(x)}{V^{(q)}(c_{2})}\upsilon_{c_{1},c_{2}}^{\kappa^{r}}(c_{2}), 
\end{equation}
where last equality follow from (\ref{Two sided exit problem Parisian}). If we are at point $c_{2}$ we paid $c_{2}-c_{1}-\beta$ and decrease $U^{c_{1},c_{2}}$ by $c_{2}-c_{1}$. Again by strong Markov property we have:
\begin{equation}
\nonumber
\upsilon_{c_{1},c_{2}}^{\kappa^{r}}(c_{2}) = c_{2}-c_{1}-\beta + \upsilon_{c_{1},c_{2}}^{\kappa^{r}}(c_{1}) = c_{2}-c_{1}-\beta + \frac{V^{(q)}(c_{1})}{V^{(q)}(c_{2})}\upsilon_{c_{1},c_{2}}^{\kappa^{r}}(c_{2}).
\end{equation}
Next step is to just solve above equation with respect to $\upsilon_{c_{1},c_{2}}^{\kappa^{r}}$. We obtain:
\begin{equation}
\nonumber
\upsilon_{c_{1},c_{2}}^{\kappa^{r}}(c_{2}) = \frac{V^{(q)}(c_{2})}{V^{(q)}(c_{2})-V^{(q)}(c_{1})}(c_{2}-c_{1}-\beta)
\end{equation}
Finally, we must put above formula into $\eqref{star}$ to get the result.
\end{proof}
The idea of finding optimal points $(c_{1},c_{2})$ leads to finding minimum of the function below 
\begin{equation}
g(c_{1},c_{2}) = \frac{V^{(q)} (c_{2}) - V^{(q)} (c_{1})}{c_{2}-c_{1}-\beta}.
\end{equation}
Let us denote domain of this function as $dom(g) = \lbrace (c_{1},c_{2}): c_{1} \geq 0 , c_{2} > c_{1} + \beta \rbrace$. Let $C^{*}$ be a set of $(c_{1},c_{2})$ from $dom(g)$ that minimize function g:
\begin{equation}
\nonumber
C^{*} = \lbrace (c_{1}^{*},c_{2}^{*}) \in dom(g) : \inf_{(c_{1},c_{2}) \in dom(g)} g(c_{1},c_{2}) = g(c_{1}^{*},c_{2}^{*})\rbrace
\end{equation}
Also fix set $\mathcal{B} = \lbrace (c_{1},c_{2}) : (c_{1},c_{2}) \in dom(g), c_{1} \neq 0  \rbrace$
\begin{prop}\label{Prop derivative of v at c2}
For $W^{(q)} \in C^{1}(0,\infty)$ the set $C^{*}$ is not empty and for each $(c_{1}^{*},c_{2}^{*}) \in C^{*}$ we have
\begin{equation}\label{Derivative of v at c2}
V^{(q)\prime} (c_{2}^{*}) = \frac{V^{(q)} (c_{2}^{*}) - V^{(q)} (c_{1}^{*})}{c_{2}^{*}-c_{1}^{*}-\beta}.
\end{equation}
Also we know that in this case there are following possibilities : $(i)$ $V^{(q) \prime} (c_{1}^{*}) = V^{(q) \prime} (c_{2}^{*})$ or  $(ii)$ $c_{1}^{*} = 0$. \\
\begin{proof}
At the beginning we will show that if $c_{1} \rightarrow \infty$ function $g$ is not attaining its minimum.
\begin{equation}
\nonumber
\begin{split}
g(c_{1},c_{2}) =& \frac{V^{(q)} (c_{2}) - V^{(q)} (c_{1})}{c_{2}-c_{1}-\beta} = \int_{0}^{\infty}\Bigl(\frac{w^{(q)}(c_{2};-z)-w^{(q)}(c_{1};-z)}{c_{2}-c_{1}-\beta}\Bigr)\frac{z}{r}\mathbb{P}(X_{r} \in dz) 
\\ \geq & \int_{0}^{\infty}\left(\frac{W^{(q)}(c_{2}+z)-W^{(q)}(c_{1}+z)}{c_{2}-c_{1}}\right)\left(\frac{c_{2}-c_{1}}{c_{2}-c_{1}-\beta}\right)\frac{z}{r}\mathbb{P}(X_{r} \in dz) 
\\ >& \int_{0}^{\infty}\min_{x \in [c_{1}+z,c_{2}+z]}\qscaleprime (x)\frac{z}{r}\mathbb{P}(X_{r} \in dz) \geq \int_{0}^{\infty}\min_{x \in [c_{1},\infty)}\qscaleprime (x)\frac{z}{r}\mathbb{P}(X_{r} \in dz) 
\\ =&  \min_{x \in [c_{1},\infty)}\qscaleprime (x) \int_{0}^{\infty}\frac{z}{r}\mathbb{P}(X_{r} \in dz) \stackrel{c_{1} \rightarrow \infty} \longrightarrow \infty. 
\end{split}
\end{equation}
In the first inequality we used
\begin{equation}\label{Inequality on w and W}
\wq (c_{2};-z) - \wq (c_{1};-z) \geq \qscale (c_{2}+z) - \qscale (c_{1}+z).
\end{equation}
Next inequality follows from the mean value theorem $( \qscale \in C^{1}(0,\infty) )$ and the simple fact that $\frac{c_{2}-c_{1}}{c_{2}-c_{1}-\beta} > 1$. Last inequality is a consequence of 
$[c_{1}+z,c_{2}+z] \subseteq [c_{1},\infty)$ for all $(c_{1},c_{2}) \in dom(g)$ and all $z > 0$. 
Note that  $\int_{0}^{\infty} \frac{z}{r} \p (X_{r} \in dz) > 0$ and for that reason last statement follows. 
We get that, $\inf_{(c_{1},c_{2})\in dom(g)} g(c_{1},c_{2})$ is not attained when $c_{1} \rightarrow \infty$, thus we can assume that there exist $C_{1} > 0$ such that 
\[
\inf_{(c_{1},c_{2})\in dom(g)} g(c_{1},c_{2}) = \inf_{(c_{1},c_{2})\in dom(g)\wedge c_{1} \leq C_{1}} g(c_{1},c_{2})
\]
In the next step we will show the same for $c_{2}$. Namely
\begin{equation}
\nonumber
\begin{split}
\inf_{c_{1} \in [0,C_{1}]} g(c_{1},c_{2}) =&  \inf_{c_{1} \in [0,C_{1}]}
\frac{V^{(q)} (c_{2})-V^{(q)} (c_{1})}{c_{2}-c_{1}-\beta} 
\\ \geq &  \inf_{c_{1} \in [0,C_{1}]}\int_{0}^{\infty}\left(\frac{\qscale (c_{2}+z)-\qscale (c_{1}+z)}{c_{2}-c_{1}-\beta}\right)\frac{z}{r}\p (X_{r} \in dz)  
\\\geq & \Bigl( \frac{\qscale (c_{2})}{c_{2}-\beta} \int_{0}^{\infty} \frac{z}{r} \p (X_{r} \in dz)\\-&  \frac{1}{c_{2}-C_{1}-\beta} \int_{0}^{\infty} \qscale (C_{1}+z)\frac{z}{r} \p (X_{r}\in dz) \Bigr)
\stackrel{c_{2} \rightarrow \infty}\longrightarrow \infty. 
\end{split}
\end{equation}
Note that we used only (\ref{Inequality on w and W}) $($in the first inequality$)$ and the property that $\qscale$ is increasing $($in the second inequality$)$. Last step is to consider the case when $(c_{1},c_{2})$ converge to the line $c_{2} = c_{1} + \beta$. 
\begin{equation}
\nonumber
\begin{split}
g(c_{1},c_{2}) =& \int_{0}^{\infty}\left(\frac{w^{(q)}(c_{2};-z)-w^{(q)}(c_{1};-z)}{c_{2}-c_{1}-\beta}\right)\frac{z}{r}\mathbb{P}(X_{r} \in dz) 
\\ \geq &  \int_{0}^{\infty}\min_{x\in [c_{1}+z,c_{2}+z]}\qscaleprime (x)\left(\frac{\beta}{c_{2}-c_{1}-\beta}\right)\frac{z}{r}\mathbb{P}(X_{r} \in dz)\\ \geq &  \qscaleprime (a^{*}) 
\frac{\beta}{c_{2}-c_{1}-\beta} \int_{0}^{\infty} \frac{z}{r} \p (X_{r}\in dz) \rightarrow \infty.
\end{split}
\end{equation}
We used, again, mean value theorem and fact that $c_{2} > c_{1} + \beta$. We check that infimum of $g$ is not reached when $c_{1} \rightarrow \infty$ or $c_{2} \rightarrow \infty$ or $(c_{1},c_{2})$ 
converge to $c_{2} = c_{1} + \beta$. Because of it and the continuity of $g$ we get that $C^{*}$ is not empty and we are left with the following possibilities
\begin{enumerate}[(i)]
\item First is that $(c_{1}^{*},c_{2}^{*})$ belong to the interior of $\mathcal{B}$. In this case, using the fact that $g$ is partial differentiable in $c_{1}$ and $c_{2}$ $( \qscale \in C^{1}(0,\infty))$, we get that
\begin{equation}
\nonumber
\frac{\partial g(c_{1},c_{2})}{\partial c_{1}}(c_{1}^{*}) = 0 \quad \text{and} \quad \frac{\partial g(c_{1},c_{2})}{\partial c_{2}}(c_{2}^{*}) = 0.
\end{equation}
And hence  we obtain \eqref{Derivative of v at c2} and  $(i)$.
\item The second possibility is when $c_{1}^{*} = 0$.
 Then we have that $c_{2}^{*}$ minimizes function $g_{0}(c_{2}) = g(0,c_{2}) = \frac{V^{(q)} (c_{2}) - V^{(q)} (0)}{c_{2} -\beta}$. We get $(ii)$ because $g_{0}^{\prime}(c_{2}^{*}) = 0$. 
\end{enumerate} 
\end{proof}
\end{prop}
To start the optimisation reasoning, we need the following proposition and lemma.
\begin{prop}\label{optimal_cases}
Assume that $W^{(q)} \in C^{1}(0,\infty)$. Then for each $(c_{1}^{*},c_{2}^{*}) \in C^{*}$ we have that
\begin{equation}
\nonumber
\upsilon^{\kappa^{r}}_{c_{1}^{*},c_{2}^{*}}(x) = \begin{cases}
\frac{V^{(q)} (x)}{V^{(q) \prime} (c_{2}^{*})} & \text{for } x \leq c_{2}^{*}, \\ (x-c_{2}^{*}) + \frac{V^{(q)} (c_{2}^{*})}{V^{(q)\prime} (c_{2}^{*})} & \text{for } x > c_{2}^{*}.
\end{cases}
\end{equation}
\begin{proof}
From Proposition $($\ref{Prop derivative of v at c2}$)$ it follows that:
\begin{itemize}
\item for  $x \leq c_{2}^{*}$:
\begin{equation}
\nonumber
\valueP (x) = (c_{2}^{*}-c_{1}^{*}-\beta)\frac{V^{(q)}(x)}{V^{(q)}(c_{2}^{*}) - V^{(q)}(c_{1}^{*})} = \frac{V^{(q)}(x)}{V^{(q)'}(c_{2}^{*})},
\end{equation}

\item For  $ x  > c_{2}^{*}:$
\begin{equation}
\nonumber
\begin{split}
\valueP (x) &= x-c_{1}^{*}-\beta + (c_{2}^{*} -c_{1}^{*}-\beta)\frac{V^{(q)}(c_{1}^{*})}{\vq (c_{2}^{*}) - \vq (c_{1}^{*})} 
\\&= x-c_{2}^{*} + (c_{2}^{*} - c_{1}^{*}-\beta)\frac{V^{(q)}(c_{2}^{*})}{\vq (c_{2}^{*}) - \vq (c_{1}^{*})} =  x-c_{2}^{*} + \frac{\vq (c_{2}^{*})}{\vqprime (c_{2}^{*})}.
\end{split}
\end{equation}
\end{itemize} 
\end{proof}
\end{prop}
\begin{lemma}\label{inequality}
Let $(c_{1}^{*},c_{2}^{*})\in C^{*}$ and $x\geq y \geq 0$. Then:
\begin{equation}\label{lemma6}
\valueP (x) - \valueP (y) \geq x-y-\beta
\end{equation}
\begin{proof}
Note that $\valueP$ is an increasing function and because of that one can assume $x-y > \beta$. Consider the following possibilities:
\begin{enumerate}
\item for $c_{2}^{*}\leq y \leq x,$
\begin{equation}
\nonumber
\valueP (x) - \valueP (y) = x-y > x-y-\beta.
\end{equation}
\item For $y \leq  x\leq c_{2}^{*}$,
\begin{equation}
\begin{split}
\nonumber
\valueP (x) - \valueP (y) =& \frac{(c_{2}^{*}-c_{1}^{*}-\beta)(\vq (x) - \vq (y))}{\vq (c_{2}^{*}) - \vq (c_{1}^{*})} \\\geq &
\frac{(x-y-\beta)(\vq (x) - \vq (y))}{\vq (x) - \vq (y)} = x-y-\beta.
\end{split}
\end{equation}
The above inequality follows from fact that $(c_{1}^{*},c_{2}^{*})\in C^{*}$, so $(c_{1}^{*},c_{2}^{*})$ minimize function $g(c_{1},c_{2}) = \frac{\vq (c_{2}) - \vq (c_{1})}{c_{2}-c_{1}-\beta}$
\item For $y \leq c_{2}^{*} \leq x$,
\begin{equation}
\nonumber
\begin{split}
\valueP (x) - \valueP (y) & = x-c_{1}^{*}-\beta + (c_{2}^{*}-c_{1}^{*}-\beta)\left(\frac{\vq (c_{1}^{*})-\vq (y)}{\vq (c_{2}^{*})-\vq (c_{1}^{*})} \right)
\\&= x-c_{2}^{*} + (c_{2}^{*}-c_{1}^{*}-\beta)\left(1+\frac{\vq (c_{1}^{*}) - \vq (y)}{\vq (c_{2}^{*})-\vq (c_{1}^{*})}\right) 
\\&= x-c_{2}^{*} + (c_{2}^{*}-c_{1}^{*}-\beta)\left(\frac{\vq (c_{2}^{*}) - \vq (y)}{\vq (c_{2}^{*}) - \vq (c_{1}^{*})}\right)
 \\&\geq  x-y-\beta.
\end{split}
\end{equation}
Since $y\leq c_2^{*}$, the last inequality follows from point (2) with $x=c_2^{*}$.
\end{enumerate} 
\end{proof}
\end{lemma}

\subsection{Optimality}\label{optimality}
For the remainder of the paper, we will focus on verifying the optimality of the impulse
strategy at threshold level $(c_{1}^{*}, c_{2}^{*})$. The proof is led by standard Markovian arguments to show that the impulse
strategy fulfils the Verification Lemma. 
However at the beginning we will prove the following fact. 
\begin{fact}
Refracted process $R$ is a Feller process and its infinitesimal generator  
is of the form 
\begin{equation}\label{generator}
\Gamma f(x) =(\gamma-\delta\mathbf{1}_{\lbrace x >0 \rbrace})  f'(x) + \frac{1}{2}\sigma^2 f^{''}(x) - \int_0^{\infty}\left(f(x+z)- f(x)-f'(x)z\mathbf{1}_{\{0<z<1\}}\right)\Pi(dz),
\end{equation}
where $x \in \mathbb{R}$ and f is a function on $\mathbb{R}$ such that  $\Gamma f(x)$ is well defined. 
\end{fact}
\begin{proof} For  $q > 0$, $x \in \mathbb{R}$ and a non-negative or bounded measurable function $f$
define  $P_R^{(q)}f:=\mathbb{E}_x\left[\int_0^{\infty} e^{-q t} f(R_t) dt \right]$. It is sufficient to verify the following conditions:
\begin{enumerate}
\item For all $q,p >0$, $P_R^{(q)}-P_R^{(p)}=(p-q)P_R^{(q)}P_R^{(p)}$.
\item For all $q>0$, $\left\| q P_R^{(q)} 1\right\| \leq 1$.
\item For all  $q > 0$, $P_R^{(q)}$ is a map from $C_0$ to $C_0$.
\item For all $f \in C_0$, $\lim_{q \to \infty }\left\| q P_R^{(q)} f -f\right\| =0$.
\end{enumerate}
Here $C_0(\mathbb{R})$ denotes the space of continuous functions vanishing at infinity. It is a Banach space when equipped with
the uniform norm $\left\|f\right\|= \sup_{x \in \mathbb{R}} |f(x)|$. \\

Since the process $R$ is a Strong Markov process (for details see \cite{KL10}) one can observe that condition
(1) is automatically fulfilled. Condition (2) is obvious. To prove (3) and (4) the reasoning is similar as in \cite{KNKY2016} except that we need to use fluctuation identities obtained in \cite{KL10}.\\

The form of the generator follows i.a. from \cite{KDSR16} with $l(x)=\gamma-\delta\mathbf{1}_{\lbrace x >0 \rbrace}$ and $Q(x)=\sigma^2$. 
\end{proof}

\begin{lemma} [Verification Lemma]
Suppose $\hat{\pi}$ is an admissible dividend strategy such that $v_{\hat{\pi}}$ is sufficiently smooth on $\mathbb{R}$
(i.e. its first or second derivative (for $X$ of bounded or unbounded variation respectively) has at most finite number of single discontinuities), satisfies 

			\begin{align} \label{HJB-inequality}
		&(\Gamma- q)v_{\hat{\pi}}(x) \leq 0, \quad \textrm{for} \quad x \in \mathbb{R}, \\ \label{HJB-inequality_second_term}
		 &v_{\hat{\pi}}(x) - v_{\hat{\pi}}(y) \geq x-y-\beta,
		\quad \textrm{for} \quad x  \geq y.
				\end{align} 
	Then $v_{\hat{\pi}}(x)=v_{*} (x)$ for almost every $x \in \mathbb{R}$ and hence $\hat{\pi}=\pi_{*}$ is an optimal strategy.
\end{lemma}
\begin{proof}
By the definition of $v_{*}$ as a supremum, it follows that $v_{\hat{\pi}}(x)\leq v_{*}(x)$ for all $x\in\mathbb{R}$. We write $h:=v_{\hat{\pi}}$ and show that 
$h(x)\geq v_\pi(x)$ for all $\pi\in\mathcal{A}$ for all $x \in \mathbb{R}$. 
	
	Fix $\pi\in \mathcal{A}$. 
	Let $(T_n)_{n\in\mathbb{N}}$ be the sequence of stopping times defined by $T_n :=\inf\{t>0:U^\pi_t>n \textrm{ or } t-\sup{\{ s\leq t: {U}^\pi_t \geq 1/n \}}>r\}$. 
	Since ${U}^\pi$ is a semimartingale (see e.g., \cite{Schilling98}, \cite{Schnurr2012}) and $h$ is sufficiently smooth on $(0, \infty)$ we will use to the stopped process $(e^{-q(t\wedge T_n)}h({U}^\pi_{t\wedge T_n}); t \geq 0)$ the Bouleau and Yor \cite{BY81} formula for bounded variation processes and the change of variables/Meyer-It\^o's formula (cf.\ Theorem IV.71 of \cite{protter}) for unbounded variation case, and deduce that under $\mathbb{P}_x$: 
	\begin{equation*}
	\label{impulse_verif_1}
	\begin{split}
	e^{-q(t\wedge T_n)}h({U}^\pi_{t\wedge T_n})-h(x)
	= &-\int_{0}^{t\wedge T_n}e^{-qs} q h({U}^\pi_{s-}) \mathrm{d}s + \frac{\sigma^2}{2}\int_0^{t\wedge T_n}e^{-qs}h''({U}^\pi_{s-})\mathrm{d}s \\
	&+\int_{[0, t\wedge T_n]}e^{-qs}h'({U}^\pi_{s-}) \mathrm{d}  ( R_s- {L}^\pi_s  ) \\	
	& + \sum_{0 \leq s\leq t\wedge T_n}e^{-qs}[\Delta h({U}^\pi_{s-}+\Delta R_s)-h'({U}^\pi_{s-})  \Delta R_s  ],
	\end{split}
	\end{equation*}
	where we use the following notation: $\Delta \zeta(s):= \zeta(s)-\zeta(s-)$ and $\Delta h(\zeta(s)):=h(\zeta(s))-h(\zeta(s-))$ for any  process $\zeta$ with left-hand limits. 
	Rewriting the above equation leads to 
	\begin{equation*}
	\begin{split}
	e^{-q(t\wedge T_n)}h({U}^\pi_{t\wedge T_n})  -h(x)
	&=   \int_{0}^{t\wedge T_n}e^{-qs}   (\Gamma-q)h({U}^\pi_{s-})   \mathrm{d}s
	-\int_{0}^{t\wedge T_n}e^{-qs}h'({U}^\pi_{s-})\mathrm{d}{L}^\pi_s   
	 \\&+ M_{t \wedge T_n},
	 \end{split}
	\end{equation*}
	where $\{M_t: t\geq 0\}$ is a zero-mean $\mathbb{P}_x$-martingale. 
	 Hence  using  the assumptions \eqref{HJB-inequality}, \eqref{HJB-inequality_second_term} we obtain that 
	\begin{equation} \label{w_lower}
	\begin{split}
	h(x) \geq &
	\int_{0}^{t\wedge T_n}e^{-qs} \mathrm{d}\left({L}^\pi_s -\beta \sum_{0\leq z\leq s} \mathbf{1}_{\lbrace \Delta {L}^\pi_z >0\rbrace} \right) - M_{t\wedge T_n}+ e^{-q(t\wedge T_n)}h({U}^\pi_{t\wedge T_n}).
	\end{split}
	\end{equation}
	
Now, taking expectations  in \eqref{w_lower}, using the fact that $(M_{t \wedge T_n}:t\geq0 )$ is a zero-mean $\mathbb{P}_x$-martingale and $h\geq 0$,
	letting $t$ and $n$ go to infinity ($T_n\xrightarrow{n \uparrow \infty} \kappa^{r}$ $\mathbb{P}_x$-a.s.),  and the dominated convergence gives	
	\begin{align*}
	h(x) \geq &\lim_{t,n\uparrow\infty}\mathbb{E}_x \left[ \int_{0}^{t\wedge T_n}e^{-qs} \mathrm{d}\left({L}^\pi_s -\beta \sum_{0\leq z\leq s} \mathbf{1}_{\lbrace \Delta {L}^\pi_z >0\rbrace} \right)
	- M_{t\wedge T_n}+ e^{-q(t\wedge T_n)}h({U}^\pi_{t\wedge T_n})\right] \\
	\geq &  \mathbb{E}_x \left[ \int_{0}^{\kappa^{r}}e^{-qs} \mathrm{d}\left({L}^\pi_s -\beta \sum_{0\leq z\leq s} \mathbf{1}_{\lbrace \Delta {L}^\pi_z >0\rbrace} \right) + \lim_{t,n\uparrow\infty} e^{-q(t\wedge T_n)}h({U}^\pi_{t\wedge T_n})\right]
	 \geq v_\pi(x),
	\end{align*}
which completes the proof.
	\end{proof}

\begin{remark}\label{smoothness}
The lemma presented below requires 
some smoothness on the NPV of a $(c_1; c_2)$ policy. In the view of Proposition \ref{optimal_cases} it means that some smoothness conditions on the scale function $V^{(q)}$ are required. 
We will call the scale function $V^{(q)}$ sufficiently smooth if $W^{(q)} \in C^1(0,\infty)$ when $X$ is of bounded
variation. From Theorem 2.9 of \cite{KRS2008} one can see that a necessary and sufficient condition for this is that the L\'evy measure
has no atoms. When X is of unbounded variation
we call the scale function $V^{(q)}$ sufficiently smooth if $W^{(q)} \in C^1(0,\infty)$ and $W^{(q)'}$
is absolutely continuous on $(0,\infty)$ with a density which is bounded on sets of
the form $[1/n, n]$, $n \geq  1$. Moreover, in Theorem 2.6 of \cite{KRS2008} it is proved that $W^{(q)} \in C^2(0,\infty)$ if the Gaussian coefficient $\sigma$ is strictly positive. Note that the term sufficiently smooth is used here in a slightly weaker sense (which is  explained in detail in the Lemma below). 
 
\end{remark}

\begin{lemma} \label{V_HJB}
If $V^{(q)}$ is sufficiently smooth and fulfils \eqref{HJB-inequality} and \eqref{HJB-inequality_second_term}, then 
$v^{\kappa^{r}}_{c_1^{*}, c_2^{*}}=v_{*}$ for almost every $x \in \mathbb{R}$. 
\end{lemma}
\begin{proof}
From Lemma \ref{inequality} one can see that it is sufficient to prove that \eqref{HJB-inequality} holds. 
At first to get that $(\Gamma- q)v^{\kappa^{r}}_{c_1^{*}, c_2^{*}}= 0$, for  $x<c_2^{*},$ one can observe that from Proposition \ref{optimal_cases} (for $x<c_2^{*}$)  it is enough to show that $(e^{-q(t \wedge \kappa^r \wedge \kappa_c^+)} V^{(q)}(R_{t \wedge \kappa^r \wedge \kappa_c^+}))_{t \geq 0}$ is a $\mathbb{P}_x$-martingale.
Indeed, let $\tau:=\kappa^r \wedge \kappa_c^+$, using \eqref{Two sided exit problem Parisian} together with fact that $V^{(q)}(R_{\tau})/V^{(q)}(c)=\mathbf{1}_{\{\kappa_c^+ < \kappa^r \}}$,
one can get 
\begin{align*}
\mathbb{E}_x\left[ e^{-q\tau} V^{(q)}(R_{\tau}) | \mathcal{F}_t \right]=& \mathbf{1}_{\{t\leq \tau\}} e^{-qt} \mathbb{E}_{R_t}\left[ e^{-q\tau} V^{(q)}(R_{\tau}) \right]
+ \mathbf{1}_{\{\tau<t\}} e^{-q\tau} V^{(q)}(R_{\tau}) \\=& 
\mathbf{1}_{\{t\leq \tau\}} e^{-qt}  V^{(q)}(R_{\tau}) 
+ \mathbf{1}_{\{\tau<t\}} e^{-q\tau} V^{(q)}(R_{\tau})\\=&
e^{-q( t \wedge\tau)} V^{(q)}(R_{t \wedge\tau}). 
\end{align*}
From Proposition \ref{Derivative of w exist} the derivative of scale function $V^{(q)}$ does not exist at $0$, when $X$ is of bounded
variation. Moreover, when $\sigma>0$ the second left-derivative of $v^{\kappa^{r}}_{c_1^{*}, c_2^{*}}$ at $c_2^{*}$ does not equal zero and hence $(\Gamma-q)v^{\kappa^r }_{c_1^{*}, c_2^{*}}$ is not well defined. 
Therefore we claim that the result below holds for almost every $x \in \mathbb{R}$. Indeed, it is sufficient to show that for any $t>0$

\begin{equation}\label{U_fala}
\int_0^t e^{-q s} (\Gamma -q) v^{\kappa^{r}}_{c_1^{*}, c_2^{*}}(\tilde{U}_s^{c_1^{*}, c_2^{*}}) ds \leq 0
\end{equation}
almost surely, where $\tilde{U}^{c_1^{*}, c_2^{*}}$ is the right-continuous modification of $U^{c_1^{*}, c_2^{*}}$.
One can prove it using the occupation formula for the semi-martingale local time (see e.g. \cite{protter}, Corollary 1,
p.219). For details see Lemma 6 in \cite{L08}, where the case $\sigma>0$ was considered.
Since process of bounded variation is a quadratic pure jump semimartingale (see e.g. \cite{protter}, Theorem 26,
p.71) then \eqref{U_fala} automatically holds. 
\end{proof} 

\begin{theorem}\label{optimality theorem}
Suppose that $V^{(q)}$ is sufficiently smooth and that there exists $(c_1^{*}, c_2^{*}) \in C^{*}$
such that 
\begin{equation}\label{V_prime}
V^{(q)'}(x) \leq V^{(q)'}(y) \quad \textrm{ for all} \quad c_2^{*} \leq x \leq y. 
\end{equation}
Then the strategy $\pi_{c_1^{*}, c_2^{*}}$ is an optimal strategy for the impulse control problem.
\end{theorem}
\begin{proof}
From Lemma \ref{inequality} one can see that it is sufficient to prove that \eqref{HJB-inequality} holds. 
At first, from the proof of Lemma \ref{V_HJB} we obtain that $(\Gamma- q)v^{\kappa^{r}}_{c_1^{*}, c_2^{*}}= 0$, for  $x<c_2^{*}$.
On the other hand, if $x>c_2^{*}$ we get that $(\Gamma- q)v^{\kappa^{r}}_{c_1^{*}, c_2^{*}}\leq 0$. This follows from the Proposition 
\ref{optimal_cases}  which gives that $v^{\kappa^{r}}_{c_1^{*}, c_2^{*}}=v^{\kappa^{r}}_{c_2^{*}}$, where is the value of the barrier strategy at level
$c_2^{*}$ in the de Finetti problem and fact that
$$\lim_{y\uparrow x} (\Gamma- q)(v^{\kappa^{r}}_{ c_2^{*}}-v^{\kappa^{r}}_{x}(y))\leq 0 \quad \textrm{for $x>c_2^{*}$} .$$
The above inequality one can prove using ideas from \cite{L08}[Theorem 2] together with \eqref{V_prime}. 
\end{proof}

\section{Examples}
In this part, we will present the results concerning the numerical calculations of the optimal impulse policy $(c_1^{*}, c_2^{*})$. From Proposition \ref{first_representation} we know that when $C^{*}$ is not an empty set, then $(c_1^{*}, c_2^{*})$ need to satisfy one of the possibilities listed there. Such an observation will define the way of constructing numerical calculations. However, to even start the computations one need to know how to calculate Parisian refracted scale function. 
Therefore, we will find analytical representation for $\wq$ and $\vq$ for the linear Brownian motion and for the Cr\'amer-Lundberg process with the exponential claims. 
Moreover, we will prove that for these two processes there is a unique $(c_1, c_2)$ policy which is optimal for the impulse control problem. 

\subsection{Linear Brownian motion}\hfill\\
Let us assume that process $X$ is a linear Brownian motion, which can be represented as
\begin{equation}
\nonumber
X_{t} = \mu t + \sigma B_{t},
\end{equation}
where $\mu \in \mathbb{R}$ and $\sigma > 0$. Fix $q > 0$ and $\delta > 0$. Recall that (see, e.g. \cite{IC16})
\begin{equation}
\nonumber
\qscale (x) = \frac{2}{\sigma^2 \rho}\Bigl(e^{\rho_{2} x} - e^{-\rho_{1} x}\Bigr)
\end{equation}
\begin{equation}
\nonumber
\WWq (x) =  \frac{2}{\sigma^2 \rho^{Y}}\Bigl(e^{\rho_{2}^{Y}x}-e^{-\rho_{1}^{Y}x}\Bigr),
\end{equation}
where
\begin{equation}
\nonumber
\rho_{1} = \frac{\sqrt{\mu^{2} + 2 q \sigma^2}+\mu}{\sigma^2}, \quad \rho_{2} =  \frac{\sqrt{\mu^{2} + 2 q \sigma^2}-\mu}{\sigma^2}, \quad \rho = \rho_{1} + \rho_{2} =  \frac{2\sqrt{\mu^{2} + 2 q \sigma^2}}{\sigma^2}.
\end{equation}
and
\begin{equation}
\begin{split}
\nonumber
&\rho_{1}^{Y} = \frac{\sqrt{(\mu-\delta)^{2} + 2 q \sigma^2}+(\mu-\delta)}{\sigma^2}, \quad \rho_{2}^{Y} =  \frac{\sqrt{(\mu-\delta)^{2} + 2 q \sigma^2}-(\mu-\delta)}{\sigma^2}, \\& \rho^{Y} = \rho_{1}^{Y} + \rho_{2}^{Y} =  \frac{2\sqrt{(\mu-\delta)^{2} + 2 q \sigma^2}}{\sigma^2}.
\end{split}
\end{equation}
Our first step is to present the formula for $\wq$.

\begin{prop}\label{wq_brown}
For the linear Brownian motion the function $\wq$ is of the following form
\begin{equation}
\nonumber
\begin{split}
 \wq (x;-z) &=  \frac{\sigma^2}{2} \qscaleprime (z) \WWq (x) + \frac{\qscale (z)}{\rho^{Y}}\Bigl(\rho_{1}^{Y}e^{\rho_{2}^{Y}x}+\rho_{2}^{Y}e^{-\rho_{1}^{Y}x}\Bigr) \\& = \frac{\sigma^2}{2} \qscaleprime (z) \WWq (x)  + \frac{\qscale (z)}{2}\Bigl(e^{\rho_{2}^{Y}x}+e^{-\rho_{1}^{Y}x}\Bigr) + \frac{\mu-\delta}{2} \qscale (z) \WWq (x)
\end{split}
\end{equation}
\end{prop}
\begin{proof}
The proof contains simple calculations which involves the following relations between parameters of $\qscale$ and $\qscaleY$
\begin{equation}\label{Brownian Motion - parameters relation 1}
\begin{split}
\frac{\rho_{2}}{\rho_{2}-\rho_{2}^{Y}} - \frac{\rho_{2}}{\rho_{1}^{Y}+\rho_{2}} = -\frac{\sigma^2\rho^{Y}}{2\delta}, \quad \quad \quad  \frac{\rho_{1}}{\rho_{2}^{Y}+\rho_{1}} + \frac{\rho_{1}}{\rho_{1}^{Y} - \rho_{1}} = -\frac{\sigma^2 \rho^{Y}}{2\delta}
\end{split}
\end{equation}
\end{proof}

Now we will consider the formula for the function $\vq$.
\begin{prop}
In the linear Brownian motion setting function $\vq$ is of the following form.
\begin{itemize}
\item For $x \geq 0$ 
\begin{equation}
\begin{split}
\nonumber
\vq (x) &= \frac{\sigma^2}{2} \WWq (x) \Bigl[ \frac{2}{\sqrt{2\pi \sigma^2 r}} e^{\frac{-r\mu^2}{2\sigma^2}}+ \rho_{2}e^{qr} - \rho e^{qr} \Phi\Bigl(\frac{-r\sqrt{\mu^2 +2q\sigma^2}}{\sigma\sqrt{r}}\Bigr)\Bigr] \\&+ \frac{e^{qr}}{\rho^{Y}} \Bigl(\rho_{1}^{Y} e^{\rho_{2}^{Y}x}+ \rho_{2}^{Y}e^{-\rho_{1}^{Y}x}\Bigr)
\end{split}
\end{equation}
\item For $x < 0$
\begin{equation}
\begin{split}
\nonumber
\vq (x) =& e^{qr}\Bigl(e^{\rho_{2} x}\Bigl[1-\Phi\Bigl(\frac{-x-r\sqrt{\mu^2+2q\sigma^2}}{\sigma \sqrt{r}}\Bigr)\Bigr] \\& + e^{-\rho_{1} x}\Bigl[1-\Phi\Bigl(\frac{-x+r\sqrt{\mu^2+2q\sigma^2}}{\sigma\sqrt{r}}\Bigl)\Bigr]\Bigr),
\end{split}
\end{equation}
\end{itemize}
where $\Phi$ is the cumulative distribution function of the standard normal variable.
\end{prop}
\begin{proof}
We will separate our proof into two parts
\begin{itemize}
\item For $x \geq 0$. Using formula for the $\wq$ from the last proposition one can get
\begin{equation}
\nonumber
\begin{split}
\vq (x) &=  \vqint = \frac{\sigma^2}{2} \WWq (x) \int_{0}^{\infty} \qscaleprime (z) \frac{z}{r}\mathbb{P}(X_{r} \in dz)  \\& +  \frac{\Bigl(\rho_{1}^{Y} e^{\rho_{2}^{Y}x} + \rho_{2}^{Y} e^{-\rho_{1}^{Y}}\Bigr)}{\rho^{Y}} \int_{0}^{\infty} \qscale (z) \frac{z}{r} \mathbb{P}(X_{r} \in dz). 
\end{split}
\end{equation}
Hence, one need to calculate two integrals
\begin{equation}\nonumber
\int_{0}^{\infty} \qscale (z) \frac{z}{r} \mathbb{P}(X_{r} \in dz)
\end{equation}
and
\begin{equation}\nonumber
 \int_{0}^{\infty} \qscaleprime (z) \frac{z}{r}\mathbb{P}(X_{r} \in dz).
\end{equation}
For the first integral, one can use $(\ref{int wq and Xr in dz})$. However, for the second integral we need to do some calculations. One can get the following
\begin{equation}
\nonumber
\begin{split}
&\int_{0}^{\infty} \qscaleprime (z) \frac{z}{r}\mathbb{P}(X_{r} \in dz)= \frac{2}{\sqrt{2\pi \sigma^2 r}}e^{\frac{-r\mu^2}{2\sigma^2}} + \rho_{2} e^{qr} - \rho e^{qr}\Phi\Bigl(\frac{-r\sqrt{\mu^2 + 2q\sigma^2}}{\sigma\sqrt{r}}\Bigr).
\end{split}
\end{equation}
Therefore, for $x \geq 0$ the formula for the Parisian refracted scale function is of the following form
\begin{equation}
\begin{split}
\nonumber
\vq (x) &= \frac{\sigma^2}{2} \WWq (x) \Bigl[ \frac{2}{\sqrt{2\pi \sigma^2 r}} e^{\frac{-r\mu^2}{2\sigma^2}}+ \rho_{2}e^{qr} - \rho e^{qr} \Phi\Bigl(\frac{-r\sqrt{\mu^2 +2q\sigma^2}}{\sigma\sqrt{r}}\Bigr)\Bigr] \\&+ \frac{e^{qr}}{\rho^{Y}} \Bigl(\rho_{1}^{Y} e^{\rho_{2}^{Y}x}+ \rho_{2}^{Y}e^{-\rho_{1}^{Y}x}\Bigr)
\end{split}
\end{equation}
\item Let us assume that $x < 0$. Then
\begin{equation}
\begin{split}
\nonumber
\vq (x) &= \int_{0}^{\infty} \qscale (x+z) \frac{z}{r} \mathbb{P}(X_{r} \in dz) \\&= \int_{-x}^{\infty} \frac{2}{\sigma^2\rho} \Bigl(e^{\rho_{2}(x+z)} - e^{-\rho_{1} (x+z)}\Bigr) \frac{z}{r} \frac{1}{\sqrt{2\pi \sigma^2 r}} e^{\frac{-(z-\mu r)^2}{2\sigma^2 r}} dz
\end{split}
\end{equation}
Therefore, after some calculations one can get that
\begin{equation}
\begin{split}
\nonumber
\vq (x) &= e^{qr}\Bigl(e^{\rho_{2} x}\Bigl[1-\Phi\Bigl(\frac{-x-r\sqrt{\mu^2+2q\sigma^2}}{\sigma \sqrt{r}}\Bigr)\Bigr] \\&+ e^{-\rho_{1} x}\Bigl[1-\Phi\Bigl(\frac{-x+r\sqrt{\mu^2+2q\sigma^2}}{\sigma\sqrt{r}}\Bigl)\Bigr]\Bigr)
\end{split}
\end{equation}
\end{itemize}
\end{proof}

\begin{prop}
Fix any $q > 0$ and $z > 0$, there  exist a constant $a^{*}_R \geq 0$ such that the function $\wqprime(x;-z)$ is decreasing on $(0,a^{*}_R)$ and is increasing on $(a^{*}_R,\infty)$. This also implies the same for $\vqprime(x)$.
\end{prop}
\begin{proof}
To prove the thesis, we will examine the second derivative with respect to $x$ of $\wqprime(x;-z)$. 
Indeed, using Proposition \ref{wq_brown} and the above explicit formula for the scale function $\WWq$, we get 
\begin{equation}
\nonumber
\begin{split}
 w^{(q)''} (x;-z) &=  \frac{(\rho_{2}^{Y})^2 }{\rho^{Y}}e^{\rho_{2}^{Y}x}\Bigl(  \qscaleprime (z)+\rho_{1}^{Y} W^{(q)}(z)\Bigr) - \frac{(\rho_{1}^{Y})^2 }{\rho^{Y}}e^{-\rho_{1}^{Y}x}\Bigl(  \qscaleprime (z)-\rho_{2}^{Y} W^{(q)}(z)\Bigr) \\&= \frac{(\rho_{2}^{Y})^2 }{\rho^{Y}}e^{\rho_{2}^{Y}x}A - \frac{(\rho_{1}^{Y})^2 }{\rho^{Y}}e^{-\rho_{1}^{Y}x}B,
\end{split}
\end{equation}
where $\rho_{1}^{Y}, \rho_{2}^{Y}>0$. 
The constant $A$ is strictly positive, because the scale function $W^{(q)}$ is increasing and strictly positive on whole positive half-line. 
Now, if $B<0$, then function $w^{(q)''} (x;-z)$ is positive for all $x,z>0$ and hence $a^{*}_R = 0$. 
If $B>0$, then  $w^{(q)''} (x;-z)$ is an increasing and unbounded function of $x$ as a sum of two increasing exponential functions. 
This completes the proof for $w^{(q)}$. For $\vqprime(x)$ we get the thesis directly from its definition.  
\end{proof}

\begin{theorem} 
For the linear Brownian motion model there is a
unique $(c_1; c_2)$ policy which is optimal for the impulse control problem. 
\end{theorem}
\begin{proof}
The thesis of the theorem follows directly from the above Proposition together with Lemma \ref{optimality theorem}. For more details see also Section 4 in \cite{L08}.
\end{proof}

Now, we will start numerical examples with the picture of $\vq$ and $\vqprime$. Let us consider the following parameters
\[
\mu = 0.5, \quad \sigma = 0.75, \quad r = 3, \quad \delta = 0.05, \quad q = 0.05 
\]
\begin{figure}[H]
\centering
\includegraphics[width=12cm]{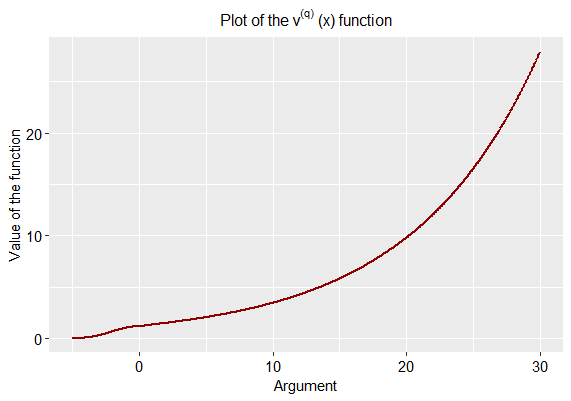}
\caption{Plot of the $\vq$ function for linear Brownian motion}
\end{figure}
Note that the shape of this function is similar as for the classic scale function for linear Brownian motion. In the  next picture we will consider $\vqprime$ with the optimal points $(c_1^*,c_2^*)$ and $\beta = 0.05$
\begin{figure}[H]
\centering
\includegraphics[width=14cm]{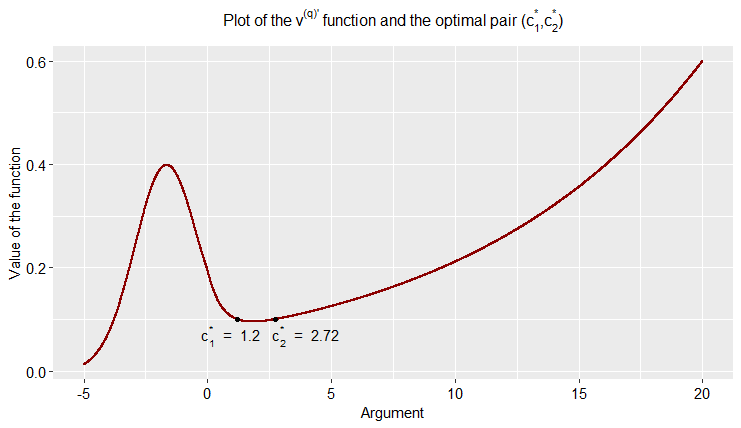}
\caption{Plot of the $\vqprime$ function for the linear Brownian motion and the optimal pair $(c_1^*,c_2^*)$. Transaction cost equal to $0.05$}
\end{figure}
The first interesting observation is the shape of this function for $x < 0$. One can see that $(c_1^*,c_2^*)$ belongs to the set $\mathcal{B}$ from the Proposition \ref{Prop derivative of v at c2}. One can also observe that our optimal pair $(c_1^{*},c_2^{*})$ satisfy condition from the Theorem \ref{optimality theorem}.   

Moreover, one can be interested in the behaviour of the optimal pair $(c_1^{*},c_2^{*})$ with the respect to the change of the parameter $\beta$. Therefore, let us set $\beta = 1$
\begin{figure}[H]
\centering
\includegraphics[width=14cm]{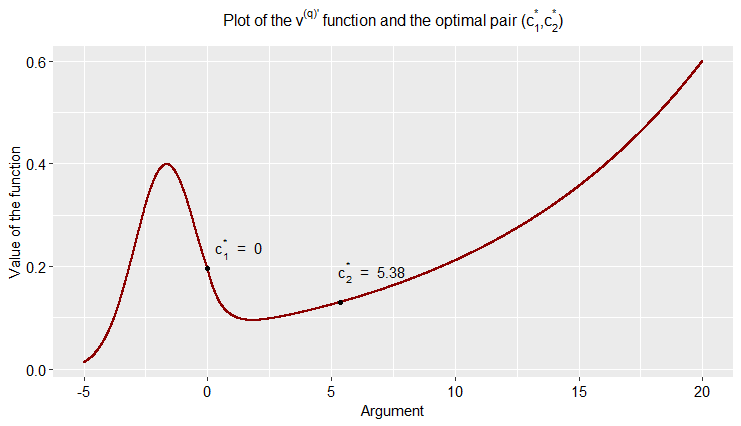}
\caption{Plot of the $\vqprime$ function for the linear Brownian motion and optimal pair $(c_1^*,c_2^*)$ not belonging to the set B. Case for $\beta = 1$}
\end{figure} 
Thus, one can see that depending on the parameters of the process one can get different possibilities from the Proposition \ref{Prop derivative of v at c2}. 
\subsection{Cram\'er-Lundberg process}\hfill\\
In the second example we will consider the Cram\'er-Lundberg process
\[
X_t = pt - \sum_{i= 1}^{N_t} U_i,
\]
where $p >0$, $\lbrace U_i \rbrace_{i=1}^{\infty}$ is an $i.i.d.$ sequence of exponential random variables with the parameter $\mu$, $\lbrace N_t \rbrace_{t \geq 0}$ is a homogeneous Poisson process with the intensity $\lambda > 0 $. We also assume that the Poisson process and the exponential random variables are mutually independent.
For this process the scale function is of the following form (see, e.g. \cite{IC16})
\[
\qscale(x) = \frac{1}{p}\Bigl(\aplus e^{\qplus x} - \aminus e^{\qminus x}\Bigr),
\]
where
\begin{equation}
\nonumber
A^{\pm} = \frac{\mu + q^{\pm}(q)}{\qplus - \qminus}, \quad \quad  q^{\pm}(q) = \frac{q + \lambda - \mu p \pm \sqrt{(q + \lambda - \mu p)^2 + 4pq \mu }}{2p}
\end{equation}
Scale function for the process $Y$ is of the form
\[
\qscaleY(x) = \frac{1}{p-\delta}\Bigl(\aplusY e^{\qplusY x} - \aminusY e^{\qminusY x}\Bigr),
\]
where 
\[
A^{\pm}_{Y} = \frac{\mu + q^{\pm}_{Y}(q)}{\qplusY - \qminusY}
\]
\[
q^{\pm}_{Y}(q) = \frac{q + \lambda - \mu(p-\delta) \pm \sqrt{(q + \lambda - \mu(p-\delta))^2 + 4(p-\delta)q \mu}}{2(p-\delta)}
\]

\begin{prop}\label{wq_CL_proces}
For $z > 0$, we have that
\begin{equation}
\nonumber
\begin{split}
\wq (x;-z) &= (p-\delta) \qscaleY(x) \qscale(z) - \frac{1}{\mu \lambda}\Bigl[(q+\lambda)\qscale(z) - p \qscaleprime(z)\Bigr]\\&\cdot \Bigl[(q+\lambda)\qscaleY (x) - (p-\delta)\mathbb{W}^{(q)'}(x)\Bigr]
\end{split}
\end{equation}
\end{prop}
\begin{proof}
To obtain such representation we had to used the following relations between the parameters of scale functions
\begin{equation}
\nonumber
\frac{\aplusY \qplus}{\qplus - \qplusY} - \frac{\aminusY \qplus}{\qplus - \qminusY} = - \frac{p-\delta}{\delta}, \quad \frac{\aminusY \qminus}{\qminus - \qminusY} - \frac{\aplusY \qminus}{\qminus - \qplusY} = \frac{p-\delta}{\delta} 
\end{equation}
and
\begin{equation}
\nonumber
\frac{\qplus}{\qplus - \qplusY} = -\frac{p-\delta}{\delta}\cdot \frac{\qplus - \qminusY}{\qplus + \mu}, \quad \frac{\qminus}{\qminus - \qplusY} = -\frac{p-\delta}{\delta}\cdot \frac{\qminus - \qminusY}{\qminus+\mu}
\end{equation}
\end{proof}

For a Parisian refracted scale function we will obtain formula, which will be divided into three parts. Nevertheless, before we state the representation we have from 
\cite{LCP13}, that
\begin{equation}\label{Measure of Ui}
\mathbb{P}(\sum_{i=1}^{N_r} U_i \in dy) = e^{-\lambda r}\Bigl(\delta_0(dy) +  e^{-\mu y} \sum_{m = 0}^{\infty} \frac{(\mu \lambda r)^{m+1}}{m!(m+1)!} y^m dy\Bigr)
\end{equation}\\ 

\begin{prop}
In the Cr\'amer-Lundberg setting the function $\vq$ is of the following form
\begin{itemize}
\item For $x > 0$
\begin{equation}
\nonumber
\begin{split}
\vq (x) =& e^{qr}(p-\delta)\qscaleY(x) - \frac{1}{\mu \lambda}\left[(q+\lambda) \qscaleY(x) - (p-\delta)\qscaleprimeY (x)\right]\\&\cdot \Bigl[(q+\lambda)e^{qr} - p C\Bigr], 
\end{split}
\end{equation}
where 
\begin{equation}
\nonumber
\begin{split}
C =  e^{-\lambda r}\Bigl[& p \qscaleprime (pr) + \aminus \qplus e^{\qplus pr}\sum_{m=0}^{\infty}\frac{(pr(\qminus+\mu))^m}{m!(m+1)!} \gamma(m+1,(\qplus+\mu)pr)\\& \cdot [pr(\qplus+\mu)-(m+1)] - \aplus\qminus e^{\qminus pr}\sum_{m=0}^{\infty}\frac{(pr(\qplus+\mu))^m}{m!(m+1)!} \\& \cdot\gamma(m+1,(\qminus+\mu)pr) \left[pr(\qminus+\mu)-(m+1)\right] \\&+ \frac{e^{-\mu p r}}{pr} \sum_{m=0}^{\infty} \frac{(p\lambda \mu r^2)^{m+1}}{m!(m+1)!} \Bigr]
\end{split}
\end{equation}
\item For $x \leq 0 \wedge x \geq -pr$
\begin{equation}
\nonumber
\begin{split}
\vq (x) =& e^{-\lambda r}\Bigl[p\qscale (x+pr) + \aminus e^{\qplus(x+pr)}\sum_{m=0}^{\infty}\frac{(pr(\qminus+\mu))^m}{m!(m+1)!}\\& \cdot\gamma(m+1,(pr+x)(\qplus+\mu)) [pr(\qplus+\mu)-(m+1)] \\&- \aplus e^{\qminus(x+pr)}\sum_{m=0}^{\infty}\frac{(pr(\qplus+\mu))^m}{m!(m+1)!}\\& \cdot  \gamma(m+1,(pr+x)(\qminus+\mu))[pr(\qminus+\mu)-(m+1)]\Bigr]
\end{split}
\end{equation}
\item For $x < -pr$
\begin{equation}
\nonumber
\vq (x) = 0
\end{equation}
\end{itemize}
where $\gamma(x,a) = \int_{0}^{x} e^{-t}t^{a-1}dt$ is an incomplete gamma function.
\end{prop}
\begin{proof}
We will divide this proof into the following parts
\begin{itemize}
\item For $x > 0$
\begin{equation}
\begin{split}
\nonumber
\vq (x) =& \vqint = (p-\delta)\qscaleY(x)\int_{0}^{\infty}\qscale (z) \frac{z}{r} \mathbb{P}(X_r \in dz) -\\& \frac{1}{\mu \lambda}\Bigl[(q+\lambda)\qscaleY(x) - (p-\delta)\qscaleprimeY(x)\Bigr]\Bigl[(q+\lambda) \int_{0}^{\infty}\qscale(z)\frac{z}{r}\mathbb{P}(X_r \in dz) - \\& p\int_{0}^{\infty}\qscaleprime(z)\frac{z}{r}\mathbb{P}(X_r \in dz)\Bigr] 
\end{split}
\end{equation}
One can see that we need to calculate the following integrals
\begin{equation*}
 \int_{0}^{\infty}\qscale(z)\frac{z}{r}\mathbb{P}(X_r \in dz)
\end{equation*}
and
\begin{equation*}
\int_{0}^{\infty}\qscaleprime(z)\frac{z}{r}\mathbb{P}(X_r \in dz).
\end{equation*}
For the first integral one can use (\ref{int wq and Xr in dz}), however for the second integral one need to do some calculations with the use of (\ref{Measure of Ui}). Therefore,
\begin{equation}
\begin{split}
\nonumber
\int_{0}^{\infty}\qscaleprime (z) \frac{z}{r}\mathbb{P}(X_r \in dz) =& e^{-\lambda r}\Bigl[ p \qscaleprime (pr) + \aminus \qplus e^{\qplus pr}\sum_{m=0}^{\infty}\frac{(pr(\qminus+\mu))^m}{m!(m+1)!}\cdot \\&\cdot \gamma\left(m+1,(\qplus+\mu)pr\right)\left[pr(\qplus+\mu)-(m+1)\right] \\&- \aplus\qminus e^{\qminus pr}\sum_{m=0}^{\infty}\frac{(pr(\qplus+\mu))^m}{m!(m+1)!}  \\& \cdot \gamma\left(m+1,(\qminus+\mu)pr\right)\left[pr(\qminus+\mu)-(m+1)\right] \\&+ \frac{e^{-\mu p r}}{pr} \sum_{m=0}^{\infty} \frac{(p\lambda \mu r^2)^{m+1}}{m!(m+1)!}\Bigr]
\end{split}
\end{equation}
Putting all the pieces together one can get postulated formula for $\vq$ for $x > 0.$

\item For $x \leq 0 \wedge x \geq -pr$\\
Note that, in this case, $\vq$ is of the following form 
\[
\vq (x) = \int_{0}^{\infty} \qscale (x+z) \frac{z}{r} \mathbb{P}(X_r \in dz)
\]
With the probability one, random variable $X_{r}$, can achieve at most value $pr$ and we know that $\qscale (x+z) > 0$ iff $x+z > 0$, thus
\[
\vq (x) = \int_{-x}^{pr} \qscale(x+z) \frac{z}{r}\mathbb{P}(X_r \in dz)
\] 
Using this observation rest of the proof involve simply, but long, calculations, thus let us omit this.
\item For $x < -pr$\\
As we state in the previous case, when $x < -pr$ then $x + z < 0$. Therefore $\qscale (x+z) = 0$ and 
\[
\vq (x) = 0
\]
\end{itemize}
\end{proof}

\begin{prop}
Fix $q > 0$ and $z > 0$, there  exist a constant $a^{*}_R \geq 0$ such that the function $\wqprime(x;-z)$ is decreasing on $(0,a^{*}_R)$ and is increasing on $(a^{*}_R,\infty)$. This also implies the same for $\vqprime(x)$.
\end{prop}
\begin{proof}
At the beginning let us note that
\begin{equation}
\frac{q+\lambda}{p-\delta} - \qplusY = \qminusY + \mu \quad \text{and} \quad \frac{q+\lambda}{p} - \qplus = \qminus+\mu .
\end{equation}
From the above and Proposition \ref{wq_CL_proces} one can obtain
\begin{equation*}
\begin{split}
&(q+\lambda)\WWq (x) - (p-\delta)\WWqprime (x) =  \frac{\mu \lambda}{(p-\delta)(\qplusY - \qminusY)}\left[e^{\qplusY x} - e^{\qminusY x}\right] \\&
(q+\lambda)\qscale(z) - p \qscaleprime(z) = \frac{\mu \lambda}{p(\qplus - \qminus)}\left[e^{\qplus z} - e^{\qminus z}\right] 
\end{split}
\end{equation*}
Therefore, one can rewrite formula for $\wq$ as 
\begin{equation*}
\begin{split}
\wq (x;-z) =& (p-\delta)\WWq (x) \qscale (z) - \frac{\mu \lambda}{p(p-\delta)(\qplus - \qminus)(\qplusY - \qminusY)}\\& \cdot \left[e^{\qplus z} - e^{\qminus z}\right]\left[e^{\qplusY x} - e^{\qminusY x}\right]
\end{split}   
\end{equation*}
From this, one can also obtain more explicit form 
\begin{equation*}
    \begin{split}
       \wq (x;-z) &= e^{\qplusY x}\Bigl[ \aplusY \qscale (z) - \frac{\mu \lambda}{p(p-\delta)(\qplus - \qminus)(\qplusY - \qminusY)}\\&\cdot\left(e^{\qplus z}  - e^{\qminus z}\right)\Bigr] - e^{\qminusY x}\Bigl[\aminusY \qscale (z) \\& - \frac{\mu \lambda}{p(p-\delta)(\qplus - \qminus)(\qplusY - \qminusY)}\left(e^{\qplus z}  - e^{\qminus z}\right)\Bigr]. 
    \end{split}
\end{equation*}
Let us fix the following notation
\begin{equation*}
    \begin{split}
       &A = \aplusY \qscale (z) - \frac{\mu \lambda}{p(p-\delta)(\qplus - \qminus)(\qplusY - \qminusY)}\left(e^{\qplus z}  - e^{\qminus z}\right),\\&
       B = \aminusY \qscale (z) - \frac{\mu \lambda}{p(p-\delta)(\qplus - \qminus)(\qplusY - \qminusY)}\left(e^{\qplus z}  - e^{\qminus z}\right).
    \end{split}
\end{equation*}
Then $\wq$ can be written as 
\[
\wq(x;-z) = Ae^{\qplusY x}-Be^{\qminusY x}.
\]
One can check that $\qplusY > 0$ and $\qminusY < 0$ and thus 
\[
\lim_{x \rightarrow +\infty}e^{\qplusY x} = +\infty, \quad \lim_{x \rightarrow +\infty}e^{\qminusY x} = 0.
\]
Then, from $\lim_{x \rightarrow +\infty} \wq(x;-z) = +\infty$ one can conclude that $A > 0$. Next, we are interested in the sign of $w^{(q)''}(x;-z)$
\[
w^{(q)''}(x;-z) = A(\qplusY)^2 e^{\qplusY x} - B(\qminusY)^2 e^{\qminusY x} 
\]
If $B < 0$ then $w^{(q)''}(x;-z)$ is positive on the whole positive half-line. In such case $a_R^{*} = 0$. If $B > 0$ then one can see that $w^{(q)''}(x;-z)$ is an increasing and unbounded function. This ends the proof.
\end{proof}

\begin{theorem} 
For the Cram\'er-Lundberg model there is a
unique $(c_1; c_2)$ policy which is optimal for the impulse control problem. 
\end{theorem}
\begin{proof}
The thesis of the theorem follows directly from the above Proposition together with Lemma \ref{optimality theorem}. For more details see also Section 4 in \cite{L08}.
\end{proof}

Using the above results one can plot the picture of the example of $\vq$ and $\vqprime$ for this process. Namely, let us set
\[
p = 3, \quad \lambda = 2, \quad \mu = 1, \quad r = 2, \quad q = 0.05, \quad \delta = 0.25
\]
Note that, we set such parameters that $p > \frac{\lambda}{\mu}$. Moreover, we know that $\vq (x) = 0$ if $x < -pr$, therefore we will consider x $\geq -pr$

\begin{figure}[H]
\centering
\includegraphics[width=14cm]{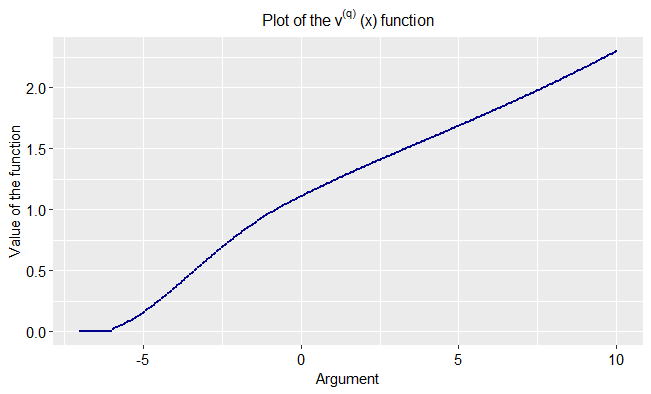}
\caption{Plot of the $\vq$ for Cram\'er-Lundberg}
\label{CLvq}
\end{figure}
As in the linear Brownian motion setting, one can also see the similar shape of Parisian scale function with the shape of classical scale function. However, even if this is not directly clear from the Figure \ref{CLvq}, $\vq$ is not a continuous function at $x = -pr$.
Now, we will also show the plot of the $\vqprime$ with the optimal points $(c_1^*,c_2^*)$ with $\beta = 0.02$
\begin{figure}[H]
\centering
\includegraphics[width=14cm]{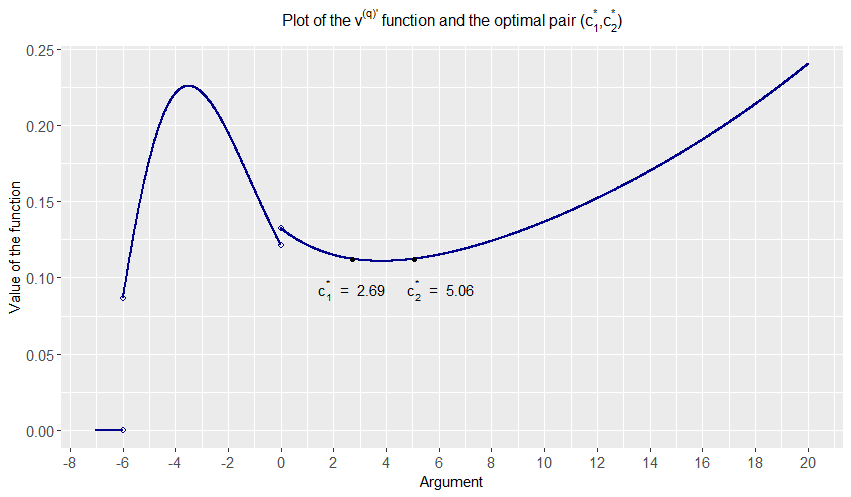}
\caption{Plot of the $\vqprime$ for the Cram\'er-Lundberg process and the optimal pair $(c_1^*,c_2^*)$. Transaction cost is equal to $0.02$}
\label{CL vq prime optimal pair v1}
\end{figure}
One can see from the Figure \ref{CL vq prime optimal pair v1} that we are in the case when $(c_1^*,c_2^*)$ belongs to the set $\mathcal{B}$. In addition, let us note that optimal pair $(c_1^{*},c_2^{*})$ satisfy the condition from the Theorem \ref{optimality theorem}. As in the case of the linear Brownian motion we would like to manipulate with the parameter $\beta$. Let us set $\beta = 1$
\begin{figure}[H]
\centering
\includegraphics[width=14cm]{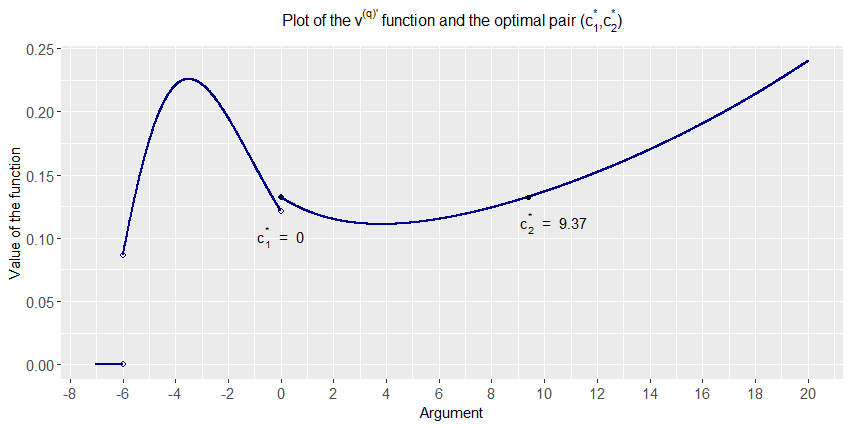}
\caption{Plot of the $\vqprime$ for the Cram\'er-Lundberg process and the optimal pair $(c_1^*,c_2^*)$ not belonging to the set $\mathcal{B}$. Case for $\beta = 1$}
\label{CL vq prime optimal pair v1 beta equal to 1}
\end{figure}
One can see from the Figure $\ref{CL vq prime optimal pair v1 beta equal to 1}$ that in such case costs of the transaction are to high and after dividend payment surplus level is moved into level zero. Note that we put point $c_1^*$ and $c_2^*$ into plot of $\vqprime$ only for illustrating purpose. It is clear that in such case we are not interested in the value of $\vqprime(0)$ because such function is not well define in this point.

\end{document}